\documentclass[12pt, letterpaper]{amsart}
\usepackage{amssymb}
\usepackage{amsmath}
\usepackage{cite}
\usepackage{graphicx}
\usepackage{float}
\usepackage{hyperref}
\usepackage{cleveref}
\usepackage[margin=1.1in]{geometry}
\usepackage{mathtools}
\usepackage{mathrsfs}
\usepackage{xcolor}

\newcommand{\CI}{\mathcal{C}^{\infty}}
\newcommand{\ep}{\epsilon}
\newcommand{\RR}{\mathbf{R}}

\DeclareMathOperator{\sgn}{sgn}

\newcommand{\phg}{\text{phg}}
\newcommand{\Lap}{\Delta}
\newcommand{\rp}{\mathcal{R}_+}
\newcommand{\rn}{\mathcal{R}_-}
\newcommand{\rpn}{\mathcal{R}_\pm}

\newcommand{\RNum}[1]{\uppercase\expandafter{\romannumeral #1\relax}}
\newcommand{\D}{\mathcal{D}}

\newcommand{\tb}{\widetilde{b}}
\newcommand{\ta}{\widetilde{a}}

\newtheorem{theorem}{Theorem}[section]
\newtheorem*{theorem*}{Theorem}
\newtheorem{corollary}{Corollary}[theorem]
\newtheorem{lemma}[theorem]{Lemma}
\newtheorem{definition}{Definition}[section]
\newtheorem{proposition}[theorem]{Proposition}
\newtheorem{remark}{Remark}[section]
\title[Polyhomogeneity, Diffraction and Scattering]{Propagation of polyhomogeneity, Diffraction and Scattering on Product Cones}

\author{Mengxuan Yang}
\date{\today}
\address{Department of Mathematics, Northwestern University}
\email{mxyang@math.northwestern.edu}
\begin{document}
\maketitle

\begin{abstract}
We consider diffraction of waves on a product cone. We first show that diffractive waves enjoy a one-step polyhomogeneous asymptotic expansion, which is an improvement of Cheeger-Taylor's classical result of half-step polyhomogeneity of diffractive waves in \cite{cheeger1982diffraction}, \cite{cheeger1982diffraction2}. We also conclude that on product cones, the scattering matrix is the diffraction coefficient, which is the principal symbol of the diffractive half wave kernel, for strictly diffractively related points on the cross section. This generalize the result of Ford, Hassell and Hillairet in 2-dimensional flat cone settings\cite{ford2018wave}. In the last section, we also give a radiation field interpretation of the relationship between the scattering matrix and the diffraction coefficient.
\end{abstract}


\section{Introduction}
In this paper, we study the diffraction coefficient of the wave equation 
$$\Box u=0$$
and the scattering matrix corresponding to the Helmholtz equation
$$\left(\Lap-\lambda^2\right)u=0$$ on a cone $C(N)$. The diffraction refers to the effect that when a propagating wave encounters a corner of an obstacle or a slit, its wave front bends around the corner of the obstacle and propagates into the geometrical shadow region. When studying the wave equation on cones, we see that its singularities likewise split into two types after they encounter the cone point. One propagates along the natural geometric extensions of the incoming ray, while other singularities emerge at the cone point and start propagating along \emph{all} outgoing directions as a spherical wave. The outgoing singularities are described to leading order by a diffraction coefficient, which is one of the central objects we study in this paper. 

The (stationary) scattering theory of the wave equation gives an approach to studying the continuous spectrum of the Laplacian on non-compact manifolds. The scattering matrix, which, intuitively speaking, maps the incoming solution at the infinity of the (stationary) wave equation to the outgoing solution, is a central object of study.   

In this paper, we focus on the diffraction and the scattering of the wave equation on cones. For notational purposes, we denote our $n$-dimensional cone by $C(N)$, which is $\mathbf{R}_+\times N^{n-1}$ with metric $dr^2+r^2h(\theta, d\theta)$ where $h(\theta, d\theta)$ is the metric on the smooth manifold $N^{n-1}$. We consider the fundamental solution to the wave equation on $C(N)$ corresponding to the Friedrichs extension of the Laplacian. For $t$ large enough, the singularities of the fundamental solutions consist of two parts by \cite{cheeger1982diffraction} \cite{cheeger1982diffraction2}. One lies on a sphere (up to reflection by the boundary of $N$) of radius $t$ to the initial point $(r',\theta')$, while the other part lies on a sphere of radius $t-r'$ around the cone point and is conormal to $\{r=t-r'\}$. We refer the latter to the \emph{diffractive wave front}, and it is the main object of our interest in this paper. The former notion will be called the \emph{geometric wave front}. See Figure \ref{pic2} in Section 2 for an example of the geometric and diffractive waves. The diffraction coefficient is therefore defined by comparing the principal symbol of the incoming wave to the principal symbol of the diffractive wave, or equivalently, reading off the principal symbol of the diffractive half wave kernel. On the other hand, the scattering matrix is defined by considering the leading order behavior to the stationary wave equation
$$\left(\Lap-\lambda^2\right)u=0$$
under certain boundary/asymptotic conditions; the solution $u$ then has the leading order behavior 
$$u\sim a_+(\theta)r^{-\frac{n-1}{2}}e^{i\lambda r}+a_-(\theta)r^{-\frac{n-1}{2}}e^{-i\lambda r}+\mathcal{O}(r^{-\frac{n+1}{2}})\ \text{ as }r\rightarrow\infty$$
where $a_+(\theta)$ is uniquely determined by $a_-(\theta)$; the scattering matrix $S(\lambda)$ is then defined by the unitary map from $a_+(\theta)$ to $a_-(\theta)$ for $\lambda\in\RR\backslash\{0\}$.
 
The following theorem thus relates two central concepts of the theories of diffraction and scattering: 
\begin{theorem}
\label{thm1.1}
Away from the intersection of the geometric wave front and diffractive wave front, the kernel of the diffractive half wave propagator is a conormal distribution of the form: 
$$
U_D(t)= (2\pi)^{-\frac{n+1}{2}}\int e^{i(r+r'-t)\cdot\lambda} K(r,\theta; r',\theta';\lambda)d\lambda |r^{n-1}drd\theta r'^{n-1}dr'd\theta'|^{\frac{1}{2}}.
$$
The principal symbol $K_0(r,r',\theta,\theta')$ of the diffractive half wave kernel $U_{D}(t,r,r',\theta,\theta')$ is related to the kernel of the scattering matrix $S(\lambda, \theta,\theta')$ by 
\begin{equation}
K_0(r, \theta, r', \theta')= (2\pi)^{-1}(rr')^{-\frac{n-1}{2}}S(\lambda, \theta,\theta'). 
\end{equation}
\end{theorem}
It is worth while to point out here that we are actually showing that the smooth part of the scattering matrix corresponds to the diffraction coefficient, while the singular part of the scattering matrix corresponds to the geometric wave. This was proved in a special case of 2-dimensional flat cones by Ford, Hassell and Hillairet\cite{ford2018wave}.

We also give a finer description of the structure of the diffractive wave by showing that it is \emph{one-step} polyhomogeneous thus improving the result of half-step polyhomogeneity of diffractive waves that appears in Cheeger-Taylor \cite{cheeger1982diffraction} \cite[Theorem 5.1, 5.3]{cheeger1982diffraction2}. Consequently, one half of the coefficients that appear in Cheeger-Taylor's expansion must vanish.
\begin{theorem}
\label{thm1.2}
The symbol $K(r,r',\theta,\theta';\lambda)$ of the diffractive half wave kernel $U_{D}(t)$ is one-step polyhomogeneous in $\lambda$ for $\lambda>0$, i.e. 
\begin{equation}
K(r,r',\theta,\theta';\lambda)\sim\sum_{i=0}^{-\infty} K_i(r,r',\theta,\theta')\lambda^{i}. 
\end{equation}
\end{theorem}

Finally, we give an interpretation of the relation in Theorem \ref{thm1.1} in terms of the radiation field. The radiation field was introduced by Friedlander\cite{friedlander2001notes} for smooth asymptotically Euclidean manifolds. It intuitively can be regarded as measuring the waves of different time delay that arrive at infinity. We define the forward radiation field $\mathcal{R}_+$ as the limit, as time goes to infinity, of the derivative of the forward fundamental solution of the wave equation along certain light rays. By reversing time one can define the backward radiation field $\mathcal{R}_-$. The forward/backward radiation field $\mathcal{R}_{\pm}$ is related to the scattering matrix $S(\lambda)$ in the following formula: 
\begin{equation}
S(\lambda)=\mathcal{F}\circ\mathcal{R}_+\circ \mathcal{R}_-^{-1}\circ\mathcal{F}^{-1}.
\end{equation}
This was first introduced by Friedlander \cite{friedlander1980radiation} in $\RR^n$, and was proved later by S{\'a} Barreto\cite{sa2003radiation} for smooth asymptotically Euclidean manifolds. The intuition and motivation of Theorem \ref{thm1.1} also come from the following facts: In principle, the scattering operator:
$$\mathscr{S}:=\mathcal{R}_+\circ \mathcal{R}_-^{-1}$$
is given by the Fourier conjugation of the leading symbol to the forward fundamental solutions \cite{friedlander2001notes}; the scattering matrix is the Fourier conjugation of the scattering operator \cite{sa2003radiation}. This suggests that the scattering matrix and the diffraction coefficient should be the same up to some constant or scaling in radial variables.

We combine Cheeger-Taylor's functional calculus on cones and Melrose-Wunsch's propagation of conormality to give a simpler calculation of the diffractive coefficient and the one-step polyhomogeneity. As for determining the scattering matrix, we consider it mode-by-mode to reduce the original equation to a Bessel equation. 

The outline of this paper is as follows. In section 2, we prove a characterization of the one-step polyhomogeneous solutions to wave equations on cones and the propagation of one-step polyhomogeneity for the diffractive wave. These will be used in section 3 to determine the diffraction coefficient. Then in section 3, we compute the diffraction coefficient using the functional calculus on cones and the propagation of conormality. In section 4, we focus on computing the scattering matrix and give the diffraction-scattering relation, Theorem \ref{thm1.1}, in the end. Finally in section 5, we give an interpretation of this result using the radiation field.

\subsection*{Acknowledgment} The author wants to thank Jared Wunsch for proposing this interesting topic and genuinely thank him for his guidance and many helpful discussions as the author's advisor. The author also wants to thank Dean Baskin, Jeremy Marzuola and Ant{\^o}nio S{\'a} Barreto for helpful discussions about the radiation field during the author's staying in MSRI.

\section{Conic Diffraction Geometry}

In this section we recall some basic notions on the geometry of product cones together with the geometric and diffractive wave on it.  

Again, we denote the product cone with total dimension $n$ by $C(N)\cong\mathbf{R}_+\times N^{n-1}$ with metric $dr^2+r^2h(\theta,d\theta)$, where $N$ is a smooth manifold and $h$ is the metric restricted to $N$. This is a particular case of general cones $\left(C(Y),\tilde{g}\right)$ with metric $\tilde{g}=dr^2+r^2\tilde{h}(r,dr,\theta,d\theta)$, where in our case the metric $h$ does not have $r$ and $dr$ dependence. Sometimes we also use $X$ to denote the product cone for simplicity when the cross section $N$ is not involved explicitly in the discussion. Without loss of generality we assume $N$ has one connected component since otherwise we can restrict to a single component. 

The Laplacian\footnote{We use the positive Laplacian throughout this paper.} on $C(N)$ is defined as 
$$\Lap=D_r^2 -i\frac{n-1}r D_r+\frac{\Delta_\theta}{r^2}, $$
where $D_r:=\frac{1}{i}\frac{\partial}{\partial r}$ is the Fourier normalization of $r$-derivative and $\Lap_{\theta}$ is the Laplacian on $N$. Here we set $\Lap$ to be the Friedrichs extension of the Laplacian acting on the domains $\CI_c(\mathring{X})$. The Friedrichs domain is defined as 
$$\mathcal{D}:=\text{Dom}(\Lap_{Fr})=\text{cl} \left\{ u\in \CI_c(X^{\mathrm{o}}): ||du||_{L^2_g(X)}+||u||_{L^2_g(X)}<\infty \right\}, $$
and $\mathcal{D}_s$ denotes the corresponding domain of $\Lap^{s/2}$. Later in this paper, we will use $L^2(\RR_t;\D_s)$ to denote the regularity on spacetime $\RR\times X$. The following proposition from \cite[Proposition 3.1]{Melrose-Wunsch1} gives a characterization of domains of Friedrichs extension: 

\begin{proposition}
\emph{(Domains of the Friedrichs extension)}
For $n>4$, 
$$\emph{Dom}(\Lap)=\left\{u\in r^{w}L^2_b(X); \Lap u\in L^2_g(X) \right\}$$
is independent of $w$ in the range $w\in(-n+2,-n/2+2)$, where $L^2_b(X)$ is the boundary weighted $L^2$-space with $r^{-\frac{n}{2}}L^2_b(X)=L^2_g(X)$. For $n=3$, the same is true for $w\in(-1,0)$. For $n=2$, 
$$\emph{Dom}(\Lap)=\left\{u\in L^2_g(X); u=c+u', c\in\mathbb{C}, u'\in r^{w}L^2_b(X),  \Lap u \in L^2_g(X) \right\}.$$
In all cases, $\Lap$ is an unbounded operator:
$$\Lap: \emph{Dom}(\Lap)\rightarrow L^2_g(X)$$
and $\emph{Dom}(\Lap)$ coincides with the domain of the Friedrichs extension. 
\end{proposition}

For a more detailed discussion we refer to \cite{Melrose-Wunsch1} and \cite{melrose2008propagation}, though it is worth pointing out the following corollary to the proposition: 

\begin{corollary}
If $u\in\mathcal{E'}(X^{\mathrm{o}})$, i.e., a compactly supported distribution in the interior of the cone $X$, then $u\in\mathcal{D}_s$ is equivalent to $u\in H^s(X)$.
\end{corollary}

The d'Alembertian acting on the spacetime $\RR\times X$ is
$$\Box=D_t^2-\Lap$$
and we also define the \emph{half-wave propagator} $U(t)$ as 
$$U(t):= e^{-it\sqrt{\Lap}}.$$

We now consider the diffraction of waves with respect to the cone point, which has been studied in detail by Cheeger and Taylor in \cite{cheeger1982diffraction} and \cite{cheeger1982diffraction2} for product cones. There are two different notions of geodesics on cones, one more restrictive than the other. We can see that these notions on product cones are special cases of \cite{Melrose-Wunsch1} on general cones with metric $dr^2+r^2h(r,\theta,d\theta)$ within a small neighborhood of the cone point. 

\begin{definition} Suppose $\gamma: (-\epsilon, +\epsilon)\rightarrow C(N)$ is a piecewise \emph{geodesic} on the cone $C(N)$ hitting the cone point only at time $t=0$, then: 
\begin{itemize}
\item The curve $\gamma$ is a \emph{diffractive geodesic} if the intermediate terminal point $\gamma(0_-)$ and the initial point $\gamma(0_+)$ lie on the boundary $\{0\}\times N$.
\item The curve $\gamma$ is a \emph{geometric geodesic} if it is a diffractive geodesic such that the intermediate terminal point $\gamma(0_-)$ and the initial point $\gamma(0_+)$ are connected by a geodesic of length $\pi$ on the boundary $\{0\}\times N$(with respect to the boundary metric $h$).
\item The curve $\gamma$ is a \emph{strictly diffractive geodesic} if it is a diffractive geodesic but not geometric geodesic. 
\end{itemize}
\end{definition}
As pointed out in \cite{Melrose-Wunsch1}, the geometric geodesics are those that are locally realizable as limits of families of ordinary geodesics in the interior $X^{\mathrm{o}}$. Figure \ref{pic1} gives geometric and diffractive geodesics at the cone point. 

\begin{figure}[H]
  \includegraphics[width=\linewidth]{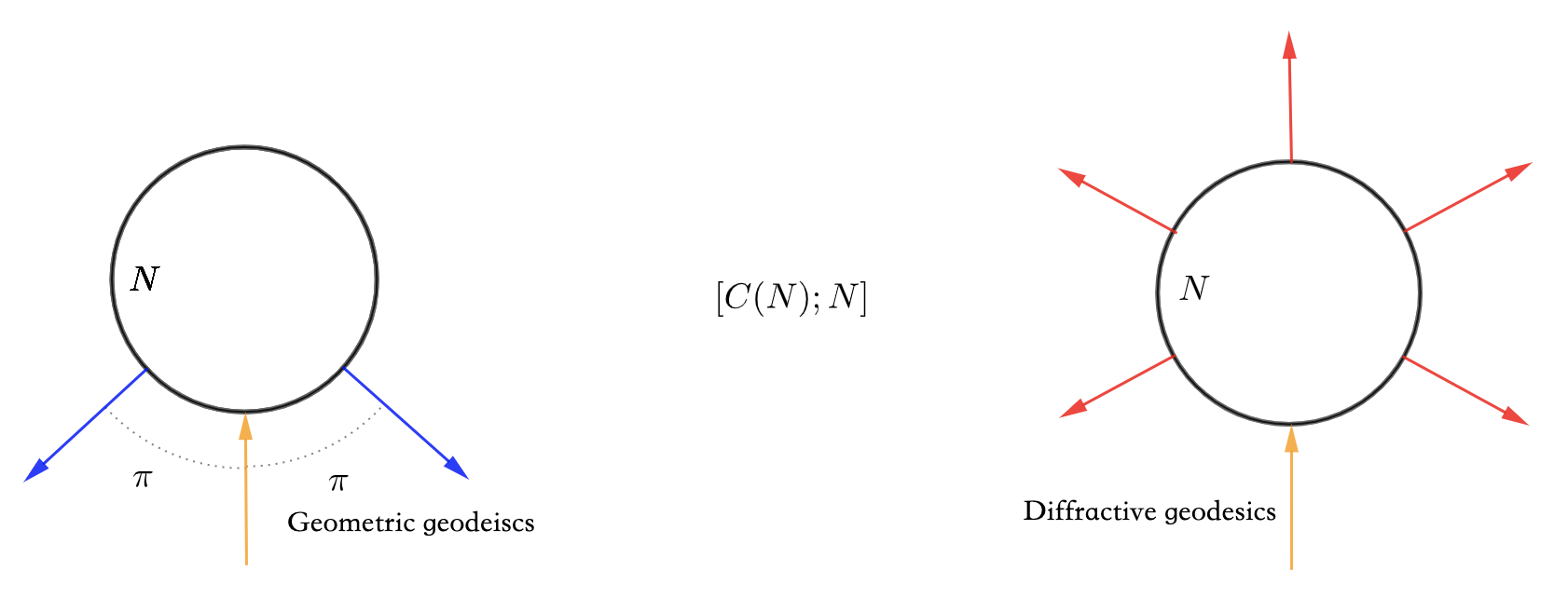}
  \caption{Diffractive and geometric geodesics}
  \medskip
  \small
  The blow up pictures of cone $C(N)$ at cone point $\{0\}\times N$:  $[C(N);N]$. On the right are the diffractive geodesics, while on the left are geometric geodesics and they are connected by geodesics in the boundary with length $\pi$.
  \label{pic1}
\end{figure}

In this paper, we focus on the diffraction coefficient and the scattering matrix away from the points that are related by the geometric geodesics, i.e. we consider the pair $(r,\theta)$ and $(r', \theta')$ with $d_{h}(\theta, \theta')\neq \pi$ for the study of the diffraction coefficient and the smooth part of scattering matrix. At the intersection of the geometric and diffractive fronts, the structure of the singularities is more complicated. This can be seen intuitively from the following picture of diffraction by a slit in Figure \ref{pic2}, which is equivalent to a product cone of angle $4\pi$. In the case of 2-dimensional flat cone, the wave kernel close to the intersection is then a singular Fourier integral operator in a calculus associated to two intersecting Lagrangian submanifolds. We refer to \cite{ford2018wave} for a detailed discussion. 

\begin{figure}[H]
  \includegraphics[width=0.7\linewidth]{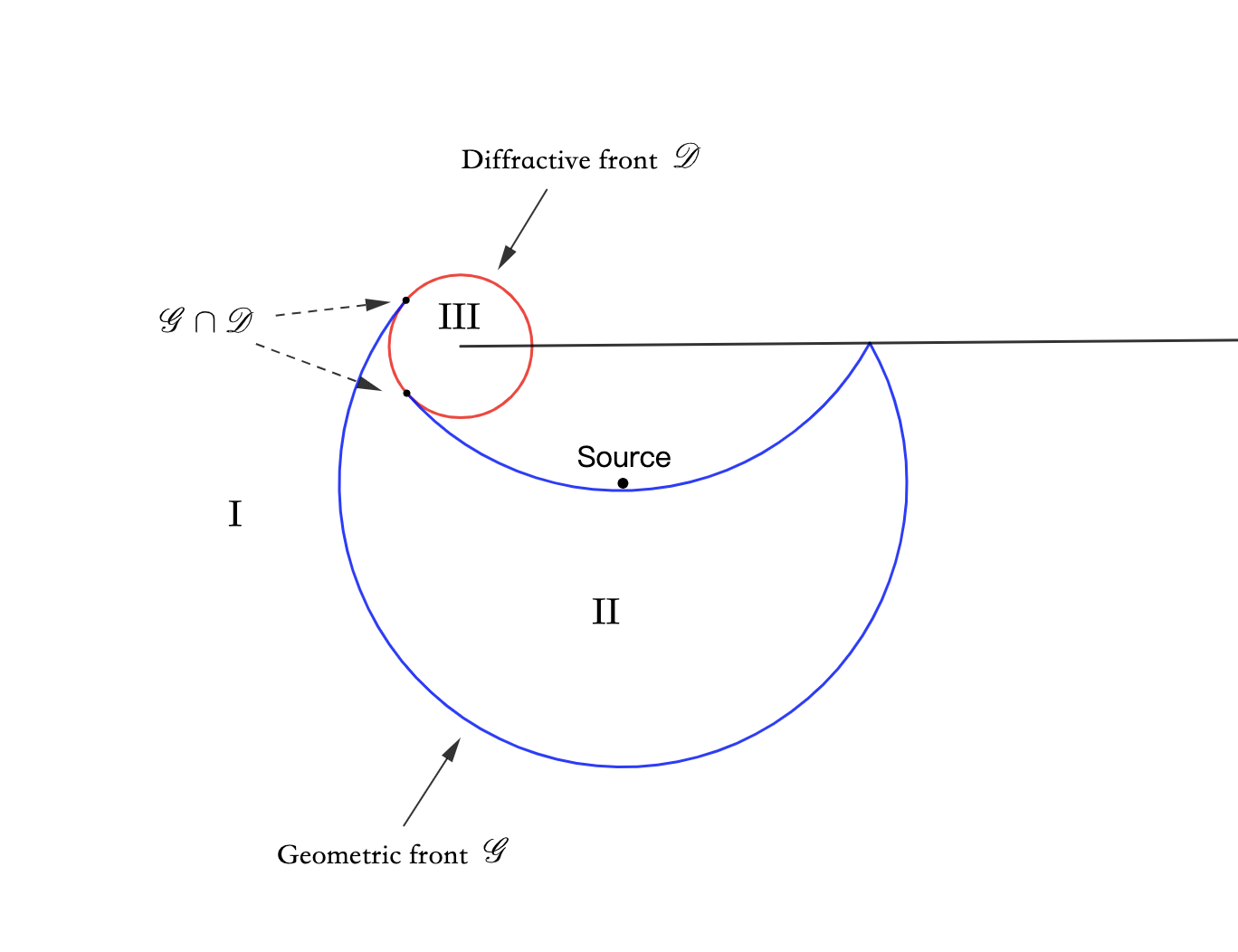}
  \caption{Geometric and Diffractive Front}
    \medskip
    \small
     Diffraction by a slit in $\RR^2$ which is a cone with cross section $N=(0, 2\pi)$. The wave source is marked on the picture. The diffractive front $\mathscr{D}$ is the boundary of region \RNum{3} (the red circle), while the geometric front $\mathscr{G}$ is part of the boundary of region \RNum{2} (the blue arcs). In this picture, the intersection of these two wave fronts $\mathscr{G}\cap\mathscr{D}$ consists of two points. 
  \label{pic2}
\end{figure}

\section{Propagation of Polyhomogeneity}
In this section, we briefly discuss the propagation of one-step polyhomogeneity on product cones in order to prove our diffractive symbol estimates. This also gives a one-step polyhomogeneity of the diffractive wave. Prior to these, we give a characterization of one-step polyhomogeneous solution to the wave equation on cones. Without loss of generality, we assume that $u$ is the spherical wave hitting the cone point at time $t=0$.

We first introduce the polyhomogeneous symbols: 

\begin{definition}[Polyhomogeneous symbols]
A symbol $a(x;\lambda)\mathcal \in {C}^{\infty}(X\times \mathbf{R}^l)$ is called (one-step) \emph{polyhomogeneous} of order $m$ if it admits an asymptotic symbol expansion:
$$a\sim\sum_{k=0}^{\infty} a_k(x,\lambda),$$
where $a_k$ are homogeneous symbols of order $m-k$, i.e., $a_k(x,\lambda)\in S^{m-k}(X; \RR^l)$ and $a_k(x, c\lambda)=c^{m-k} a_k(x,\lambda)$ for $c\in\RR_+$. We denote the polyhomogeneous symbol class of order $m$ by $S^m_{phg}(X; \RR^l)$. 
\end{definition}

We introduce the radial and tangential operators for later reference. Let 
\begin{equation}
\label{RV}
R=tD_t+rD_r
\end{equation}
denote the radial vector field on $\RR\times X$. And we also define tangential operators on $N$:
\begin{equation}
\label{TV}
Y_s=(I+\Delta_\theta)^{s/2}.
\end{equation}
Note that for the above operators and the d'Alembertian $\Box$, we have a group of commutator relations:
\begin{equation}
\label{conformalcommutator}
[\Box,R]  = -2i \Box \text{ and } [\Box, Y_s]=0
\end{equation}
which will be useful later to prove the propagation of polyhomogeneity. These commutator relations are motivated by Melrose and Wunsch's argument on propagation of conormality: from \cite[Theorem 4.8]{Melrose-Wunsch1}, these commutator relations imply that if $\Box u=0$ and $u$ is conormal to $\{t+r=0\}$ with respect to $\D_s$ for $t<0$ then $u$ is conormal to $\{t-r=0\}$ with respect to $\D_s$ for $t>0$. The conormality here is characterized by the operators $Y_s$, $R$ and $\Box$ through the following definition: 

\begin{definition}[Conormality on cones] $u\in\mathcal{D'}(\RR\times X)$ is \emph{conormal} to $\{t\pm r=0\}$ with respect to $\D_s$ if the iterative regularity: 
$$Y_iR^j\Box^k u\in L^2(\RR_t;\D_s)$$
for all $i,j,k\in\mathbb{N}$ and $t\lessgtr 0$. We use $I\D_s$ to denote this conormality, where $I$ stands for \emph{iterative}. 
\end{definition}

For a more detailed discussion on the conormal distribution, we refer to \cite{hormander2009analysis}. By the H{\"o}rmander-Melrose theory, on the product cone $X$ of total dimension $n$, the iterative regularity $I\D_{-m-1/2-\epsilon}$ for any $\epsilon>0$, with order that we will discuss later, is equivalent to the oscillatory integral definition of conormality of order $m-(n-1)/4$ which is defined in the following sense: 

\begin{definition}[Conormality again] $u\in\mathcal{D'}(\RR\times X)$ is \emph{conormal} to $\{t\pm r=0\}$ for $t\lessgtr0$ of order $k$ if it locally admits an oscillatory integral representation for $t\lessgtr 0$: 
\begin{equation}
\label{oscillatoryintegral}
u=\int e^{i(t\pm r)\lambda} a(r,\theta; \lambda) \, d\lambda\ \ \bmod \CI
\end{equation}
with Kohn-Nirenberg symbol $a(r,\theta; \lambda)\in S^{k+(n-1)/4}(X ; \RR_{\lambda})$. Let 
$$I^{k}(\RR\times X; N^* \{t\pm r=0\})$$ denote the space of all distributions on $\RR\times X$ that conormal to $\{t\pm r=0\}$ of order $k$, where $N^* \{t\pm r=0\}$ is the conormal bundle of $\{t\pm r=0\}$. 
\end{definition}

\begin{remark}
We use $ S^{m}(X; \RR_{\lambda})$ to denote the Kohn-Nirenberg symbol class on the cone $X$ of the order $m$ hereafter. 
\end{remark}

Following \cite{hormander2009analysis}, we have the following equivalent relation between the previous two definitions of conormality.

\begin{theorem}
The iterative regularity definition of conormality is equivalent to the oscillatory integral definition, more precisely, we have the following inclusions of conormal distributions on $\RR\times X$:
$$I\D_{-m-1/2}\subseteq I^{m-(n-1)/4}(\RR\times X; N^* \{t\pm r=0\}) \subseteq I\D_{-m-1/2-\epsilon}$$
for any $\epsilon>0$.
\end{theorem}

We also need the following interpolation lemma to raise the iterative conormal regularity from the Sobolev regularity and the lower iterative regularity. We refer the readers to \cite[Lemma 12.2]{melrose2008propagation} for a detailed proof of the lemma. It is presented in the form of coisotropic regularity there but the essence is the same.

\begin{lemma}[Interpolation]
\label{interpolation}
Suppose $u\in L^2(\RR_t;\D_s) \cap I\D_m$ for $s>m$, then $u\in I\D_{s-\epsilon}$ for any $\epsilon>0$.

\end{lemma}

Suppose now that $\Box u=0,$ and for $t\lessgtr 0,$
$$
u \in I^{m-(n-1)/4}(\RR\times X; N^* \{t\pm r=0\}).
$$
We can thus write, for $t\lessgtr 0,$
\begin{equation}
\nonumber
u=\int e^{i(t\pm r)\lambda} a(r,\theta,\lambda) \, d\lambda\ \ \bmod \CI
\end{equation}
for some $a\in S^{m}(X; \RR_{\lambda})$. From this, recall also that $u \in L^2(\RR_t;\D_s)$ for all $s<-m-1/2.$ We employ the notation $u \in L^2(\RR_t;\D_{-m-1/2-0})$ to denote this type of space hereafter. For the propagation of conormality, note that the symbols of $Y_1, \Box$ and $R$ consist of the defining functions of $N^* \{t\pm r=0\}$; by showing the iterative regularity: 
$$Y_iR^j\Box^ku\in L^2(\RR_t;\D_s)$$
for any $i,j,k\in\mathbb{N}$, $t\lessgtr0$ and for some order $s$, we can show that the conormality is therefore preserved. For the propagation of polyhomogeneity, we need stronger conditions. In fact, in addition to the above preservation of iterative tangential regularity, we need that applying the radial vector field with particular shifts improves the regularity by one-step at each time. We show later that this actually leads to a complete characterization of polyhomogeneous distributions. Before proceeding to the proof of the complete characterization, we start by showing a characterization of the leading order polyhomogeneity. This characterization is due to Baskin and Wunsch \cite{wunsch2019phg}. From now on, we use the notation $I\D_{s-0}$ to denote that $u$ lies in the iterative regularity class $I\D_{s-\epsilon}$ for any $\epsilon>0$.

\begin{lemma}[Characterization of the Leading Polyhomogeneity]
\label{lemma:phg}
Assume $\Box u \in \CI$, $u$ is conormal to $\{t\pm r=0\}$ and takes the oscillatory integral form \eqref{oscillatoryintegral} for $t\lessgtr0$ (away from the cone point) with $a\in S^{m}(X; \RR_{\lambda})$. Then 
$$
u \in I\D_{-m-1/2-0}.
$$
Set 
$$\alpha=m+\frac{n+1}2.$$
Then we have
\begin{enumerate}
\item  If $a$ is polyhomogeneous of order $m$, then
$$
(R-i\alpha) u \in I\D_{-m+1/2-0}.
$$

\item Conversely, suppose that
$$
(R-i\alpha) u \in L^2(\RR_t;\D_{-m+1/2-0});
$$
then
$$
a=a_m+r_{m-1}
$$
where $a_m$ is homogeneous of degree $m$ and $r_{m-1}\in S^{m-1+0}(X; \RR_{\lambda}),$ where 
$$S^{m-1+0}(X; \RR_{\lambda}):=\bigcap_{r>m-1}S^r(X; \RR_{\lambda}).$$
\end{enumerate}
\end{lemma}

\begin{proof}
Applying the operator $R-i\alpha$ to \eqref{oscillatoryintegral} and integrating by parts, we obtain
\begin{equation}\label{Ru}
  \begin{aligned}
  (R-i\alpha) u &\equiv\int e^{i(t\pm r)\lambda} \left((t\pm r)\lambda + r D_r -i\alpha) a(r,\theta,\lambda\right) \, d\lambda\\
                &\equiv\int e^{i(t\pm r)\lambda} \left(-D_\lambda \lambda + r D_r -i\alpha\right) a(r,\theta,\lambda) \, d\lambda\\
     &\equiv\int e^{i(t\pm r)\lambda} \left(-\lambda D_\lambda+ r D_r -i(\alpha-1)\right) a(r,\theta,\lambda) \, d\lambda,
  \end{aligned}
\end{equation}
here we use $\equiv$ to denote equivalence modulo $\CI$-errors.
Now we also need to use the crucial fact that $\Box u\in \CI.$  Since
$$
\Box=D_t^2-D_r^2 +i\frac{n-1}r D_r-\frac{\Delta_\theta}{r^2},
$$
applying $\Box$ to \eqref{oscillatoryintegral} yields
$$
\int e^{i (t\pm r) \lambda} (\pm 2i \lambda  \partial_r\pm i\lambda (n-1)/r + \Delta) a \, d\lambda\equiv 0 \mod \CI; 
$$
thus $a$ must satisfy the transport equation
\begin{equation}
  \label{fulltransport}
  \left(r\partial_r+\frac{n-1}{2} \mp \frac{ir}{2\lambda} \Delta \right) a \in S^{-\infty}(X; \RR_{\lambda}),
\end{equation}
where $\Lap$ is the Laplacian on cones; in particular, since $\frac{ir}{2\lambda} \Delta a \in S^{m-1}(X; \RR_{\lambda})$, this forces
$$
\left(r\partial_r+\frac{n-1}{2} \right) a\in S^{m-1}(X; \RR_{\lambda}).
$$
Plug this into the right side of \eqref{Ru} yields
$$
\begin{aligned}
  (R-i\alpha) u  &= \int e^{i(t\pm r)\lambda}\left( \left(-\lambda D_\lambda+ i\frac{n-1}{2} -i(\alpha-1)\right) a(r,\theta, \lambda)+e\right) d\lambda\\
                       &= \int e^{i(t\pm r)\lambda} \left( \left(-\lambda D_\lambda- im\right) a(r,\theta,\lambda)+e\right) d\lambda,
\end{aligned}
$$
where $e \in S^{m-1}(X; \RR_{\lambda})$ is the remainder term. 

Thus if $a\in S^m(X; \RR_{\lambda})$ is a polyhomogeneous symbol, then so is $(\lambda D_\lambda+im ) a \in S^{m-1}(X; \RR_{\lambda})$ and we find that
$$(R-i \alpha) u \in I^{m-1-(n-1)/4}(\RR\times X; N^* \{t\pm r=0\})\subset I\D_{-m+1/2-0}$$
by the equivalence of two definitions of conormality. This proves the first part of the lemma.

Conversely, if $(R-i \alpha) u \in L^2(\RR_t;\D_{-m+1/2-0}),$ by the commutator relations \eqref{conformalcommutator} and the fact that the symbol of $R$ is one of the defining functions of $\{t\pm r=0\}$, we have $(R-i \alpha)u\in I\D_{-m-1/2-0}$ by conormality of $u$. Thus by Lemma \ref{interpolation} (also see \cite[Lemma 12.2]{melrose2008propagation}), we know $(R-i \alpha)u$ is also conormal with iterative regularity $I\D_{-m+1/2-0}$. Equivalently, 
$$
(R-i \alpha) u \in I^{m-1-(n-1)/4-0}(\RR\times X; N^* \{t\pm r=0\}).
$$
This forces the symbol of $(R-i \alpha)u$ to be in the class $S^{m-1+0}(X; \RR_{\lambda})$. 
By the proof of the first part, consider the order of the symbol of $(R-i \alpha)u$ gives $$(-\lambda D_{\lambda}-im)a\in S^{m-1+0}(X; \RR_{\lambda}).$$ Equivalently, $$D_{\lambda}(\lambda^{-m}a)\in\mathcal{O}(\lambda^{-2+0}).$$ Integrating it to infinity yields 
$$\lim_{\lambda\rightarrow\infty}\lambda^{-m}a=a_m\text{ and } \lambda^{-m}a-a_m=\mathcal{O}(\lambda^{-1+0}).$$
This implies that we must have the leading asymptotic decomposition as in the statement of lemma.
\end{proof}

For later reference we record a sharpening of the symbol computation above. In particular, note that if $\Box u \in \CI$ then we can compute the symbol of $(R-i\alpha) u$ explicitly by substituting the full transport equation \eqref{fulltransport} into \eqref{Ru} to obtain
$$
(R-i\alpha) u =\int e^{i(t\pm r)\lambda}b(r,\theta,\lambda) \, d\lambda
$$
where
\begin{equation}
\label{precisetesting}
b= \left(-\lambda D_\lambda -im +\frac{r}{2\lambda} \Lap \right) a.
\end{equation}
Therefore $(R-i\alpha)$ acting on $u$ can be characterized by $(-\lambda D_\lambda -im +\frac{r}{2\lambda} \Lap)$ acting on its symbol $a$. We now generalize Lemma \ref{lemma:phg} to get a characterization of full polyhomogeneity by induction. The result given in the following lemma is similar to the characterization given by Joshi \cite{joshi1997intrinsic} for polyhomogeneous Lagrangian distributions on smooth manifolds, though the Hamilton vector field of our operator $R$ is not a multiple of the radial vector field of fiber variables as in \cite{joshi1997intrinsic}. 

\begin{lemma}[Characterization of the Complete Polyhomogeneity]
\label{lemma:fullphg}
Assume $\Box u \in \CI$, $u$ is conormal to $\{t\pm r=0\}$ and takes the oscillatory integral form \eqref{oscillatoryintegral} for $t\lessgtr 0$ with $a\in S^m(X; \RR_{\lambda})$. Set 
$$\alpha_j=m+\frac{n+1}2-j+1.$$
Then we have
\begin{enumerate}
\item  If $a$ is polyhomogeneous of order $m$, then
$$
\prod_{j=1}^k (R-i\alpha_j) u \in I\D_{-m-1/2+k-0}
$$
is conormal to $\{t\pm r=0\}$ for $t\lessgtr0$. 

\item Conversely, suppose that for the above $\alpha_1, ..., \alpha_k$, 
 $$
\prod_{l=1}^j\left(R-i\alpha_l\right) u \in L^2(\RR_t;\D_{-m-1/2+j-0});
$$
for $1\leq j\leq k$\textcolor{blue}. Then
$$
a=a_m+a_{m-1}+\cdots+a_{m-k+1}+r_{m-k}
$$
where $a_m, a_{m-1},..., a_{m-k+1}$ are homogeneous symbols with the degrees same as their indices and $r_{m-k}\in S^{m-k+0}(X; \RR_{\lambda}).$
\end{enumerate}
\end{lemma}

\begin{proof}
First assume $a\in S^m(X; \RR_{\lambda})$ is polyhomogeneous. By an integration by parts argument and the fact that $\Box u\in \CI$, we see that 
$$
\prod_{j=1}^k \left(R-i\alpha_j\right) u \in I\D_{-m-1/2+k-0}
$$
is implied by
\begin{equation}
\label{phgsymbol}
\prod_{j=1}^k \left(-\lambda D_\lambda +i\frac{n+1}{2}-i\alpha_j +\frac{r}{2\lambda} \Lap \right) a \in S^{m-k}(X; \RR_{\lambda})
\end{equation}
in the Kohn-Nirenberg class. This is aided by the commutator relation \eqref{conformalcommutator} which makes $\prod_{j=1}^k(R-i\alpha_j)u$ for all $k\in\mathbb{N}$ again a solution to the wave equation modulo smooth terms, so we can apply the transport equation substitution iteratively as we did in Lemma \ref{lemma:phg}. At each step, $\left(-\lambda D_\lambda +i\frac{n+1}{2}-i\alpha_j +\frac{r}{2\lambda} \Lap\right)$ acting on polyhomogenous symbols still gives polyhomogeneous ones. Also note that $\frac{r}{2\lambda} \Lap$ always lowers the symbol order by one, and the coefficient $\alpha_j$, when combining with $R$ to annihilate the principal part of each step, only depends on the order of the symbol it acts on. 

By Lemma \ref{lemma:phg}, we know for $\alpha_1=\frac{n+1}{2}+m$, 
$$\left(-\lambda D_\lambda +i\frac{n+1}{2}-i\alpha_1 +\frac{r}{2\lambda} \Lap \right) a \in S^{m-1} (X; \RR_{\lambda})\text{ polyhomogeneous. }$$ 
Taking this as the new symbol $b$ of the conormal distribution, then applying $(R-i\alpha_2)$ to the new distribution gives
$$\left(-\lambda D_\lambda +i\frac{n+1}{2}-i\alpha_2 +\frac{r}{2\lambda} \Lap \right) b \in S^{m-2} (X; \RR_{\lambda})\text{ polyhomogeneous. }$$
Applying this argument repeatedly with
$$
\alpha_j=\frac{n+1}{2}+(m-j+1) \text{ up to } j=k,
$$
we proved that the first statement is true.  

For the second part, we can work by induction. In the previous lemma, we have showed that $(R-i\alpha_1)$ raises regularity almost by one implies the leading one-step polyhomogeneity. Assume the conclusion in the second part of this lemma is true for up to $k$-terms one-step polyhomogeneity, i.e., assume
$$
\prod_{l=1}^j(R-i\alpha_l) u \in  L^2(\RR_t;\D_{-m-1/2+j-0})
$$
for $1\leq j\leq k$ imply 
\begin{equation}
\label{k-phg}
a=a_m+a_{m-1}+\cdots+a_{m-k+1}+r_{m-k}    
\end{equation}
with $a_m, a_{m-1},..., a_{m-k+1}$ homogeneous symbols and $r_{m-k}\in S^{m-k+0}(X; \RR_{\lambda}).$ Thus 
$$
b_k:=\prod_{j=1}^k\left(-\lambda D_\lambda +i\frac{n+1}{2}-i\alpha_j +\frac{r}{2\lambda} \Lap \right) a \in S^{m-k+0} (X; \RR_{\lambda})
$$
as in the proof of the first part. Now we consider 
$$
\prod_{j=1}^{k+1}(R-i\alpha_{j}) u \in  L^2(\RR_t;\D_{-m-1/2+k+1-0})
$$
This has the oscillatory integral form: 
$$
(R-i\alpha_{k+1})\int e^{i(t\pm r)\lambda}b_k(r,\theta,\lambda) \, d\lambda. 
$$
Now we apply Lemma \ref{interpolation}, since by the assumption
$$\prod_{j=1}^{k+1}(R-i\alpha_{j}) u \in  L^2(\RR_t;\D_{-m-1/2+k+1-0})$$
and 
$$\prod_{j=1}^{k}(R-i\alpha_{j}) u \in I\D_{-m-1/2+k-0}$$
conormal to $\{t\pm r=0\}$. Therefore, by Lemma \ref{interpolation},
$$\prod_{j=1}^{k+1}(R-i\alpha_{j}) u \in I\D_{-m-1/2+k+1-0}$$
conormal to $\{t\pm r=0\}$ and thus
$$
\left(-\lambda D_\lambda +i\frac{n+1}{2}-i\alpha_{k+1} +\frac{r}{2\lambda} \Lap\right)b_k\in S^{m-k-1+0}(X; \RR_{\lambda}).
$$
Since $\frac{r}{2\lambda} \Lap$ lowers the symbol order by one, we have: 
$$
\left(-\lambda D_\lambda +i\frac{n+1}{2}-i\alpha_{k+1} \right)b_k\in S^{m-k-1+0}(X; \RR_{\lambda}).
$$
Plugging in $\alpha_{k+1}$ defined above, using the same argument as in the proof of Lemma \ref{lemma:phg}, this forces $b_k$ to take the form $b_k=\tilde{b}_{m-k}+\tilde{r}_{m-k-1}$, where $\tilde{b}_{m-k}$ is homogeneous of order $m-k$ and $\tilde{r}_{m-k-1}\in S^{m-k-1+0}(X; \RR_{\lambda})$. Considering that the action of $\left(-\lambda D_\lambda +i\frac{n+1}{2}-i\alpha_{j} +\frac{r}{2\lambda} \Lap\right)$ for $1\leq j\leq k$ on the symbol $a$ gives $b_k$, together with $a$ being $k$-term one-step polyhomogeneous \eqref{k-phg}, $a$ must take the form
$$
a=a_m+a_{m-1}+\cdots+a_{m-k}+r_{m-k-1}
$$
where $a_m, a_{m-1},..., a_{m-k}$ are homogeneous with degrees given by their indices and $r_{m-k-1}\in S^{m-k-1+0}(X; \RR_{\lambda})$ (c.f. \cite[Proposition 2.1]{joshi1997intrinsic}). We thus finished our induction. 
\end{proof}

\begin{remark}
It is worth to point out that this lemma cannot be generalized directly to give a similar characterization of one-step polyhomogeneity on general non-product cones $\left(\RR\times Y, dr^2+r^2h(r,\theta,d\theta)\right)$. This is due to the fact that the commutator equation is now $[\Box, R]=-2i\Box+E$ with $E$ an error term. The existence of this error term makes even $(R-i\alpha_1)u$ no longer a solution to the wave equation, for which being a solution is an essential property for us to build our characterization.
\end{remark}

Recall that our ultimate goal in this section is to develop the propagation of polyhomogeneity. Prior to this we use the foregoing results to obtain a propagation of leading order polyhomogeneity which will be used later to give a diffraction symbol estimate. Then we can show the propogation of full polyhomogeneity as a corollary.

First, we recall from \cite[Theorem 4.8]{Melrose-Wunsch1} the more basic results on propagation of conormality. These follow easily in the situation at hand by commutation of $R^k$ and $Y_s$ through the equation, together with the observation that the symbols of $R,$ $Y_1,$ and $\Box$ form a set of defining functions to the conormal bundle of $\{t\pm x=0\}.$ The continuity of the evolution map asserted in the following proposition follows from the proof of \cite[Theorem 4.8]{Melrose-Wunsch1} or, as usual, from the Inverse Mapping Theorem; the essence of the direct proof is that norms of powers of the test operators are conserved relative to domains of powers of the Laplacian which agree with Sobolev spaces away from $r=0$; converting these estimates to estimates in symbol spaces requires a Sobolev embedding step, which loses at most a fixed number of derivatives (which can then be interpolated away up to an $\ep$). 
For brevity, we abbreviate the restriction to a time interval by
$$
a_{(c,d)} \equiv a\rvert_{t \in (c,d)}.
$$

\begin{proposition}[Continuity of Symbol Evolution]
\label{proposition:conormal}
Suppose that $\Box u=0$ and that for $t<0,$ $u \in I^{m-(n-1)/4}(\RR\times X; N^* \{t+r=0\}).$ Then we have:

\begin{enumerate}
\item For $t>0,$ $u \in I^{m-(n-1)/4+0}(\RR\times X; N^* \{t-r=0\}),$ thus the conormality is conserved. 
\item The map from negative time data to positive-time data is continuous in the following sense: for any $a<b<0<c<d,$ any $\ep>0,$ and any $M,$ there exist $M'$ and $C$ such that if we write $u$ in the form \eqref{oscillatoryintegral} with symbol $a$ then
  $$
||a_{(c,d)}||_{S^{m+\ep}_M}\leq C ||a_{(a,b)}||_{S^{m}_{M'}}. 
  $$
\end{enumerate}
\end{proposition}
Here the symbol norm $S^m_M$ is given by
$$
\sum_{|\alpha|+|\beta|\leq M} \sup |\langle\lambda\rangle^{-m} (\partial_{r,\theta})^\alpha (\lambda \partial_\lambda)^\beta) a |.
$$

\begin{remark} 
We will abbreviate the existence of symbol estimates of this type as ``$a$ satisfies effective estimates'' in what follows. The effective estimates here is crucial to prove the propagation of polyhomogeneity and to compute the diffraction coefficient using the mode-by-mode solutions. 
\end{remark}

\begin{proposition}[Propagation of Leading Order Polyhomogeneity]
\label{proposition:phg}
  Suppose that $\Box u=0$ and that for $t<0,$ $u \in I_{\phg}^{m-(n-1)/4}(\RR\times X; N^* \{t+r=0\}),$ i.e., conormal distributions with polyhomogeneous symbols, Then
 \begin{enumerate}
  \item For $t>0,$ $u \in I^{m-(n-1)/4}(\RR\times X; N^* \{t-r=0\})$ and has an oscillatory integral representation of the form \eqref{oscillatoryintegral} where its symbol $a$ is of the form $a_m+r_{m-1}$ with $r_{m-1} \in S^{m-1+0}(X; \RR_{\lambda}).$
  \item Moreover, for any fixed $\ep>0$, each symbol seminorm of $r_{m-1}$  in $S^{m-1+\epsilon}(X; \RR_{\lambda})$ is bounded in terms of finitely many seminorms of the symbol of $u$ for $t<0.$ 
  \end{enumerate}
\end{proposition}

\begin{proof}
  Take $\alpha$ as above.  Then for $t<0,$
  $$
(R-i\alpha) u\in I_{\phg}^{m-(n-1)/4-1}(\RR\times X; N^* \{t+r=0\})
$$
We let $b\in S^{m-1}(X; \RR_{\lambda})$ denote the total symbol of $(R-i\alpha) u,$ i.e.\
$$
(R-i\alpha)u=\int e^{i(t+ r)\lambda} b(r,\theta,\lambda)d\lambda \text{ for } t<0,
$$
while
$$
u=\int e^{i(t+ r)\lambda} a(r,\theta,\lambda)d\lambda
$$
for $a \in S^m_{phg}(X; \RR_{\lambda}).$ 

Then for $t>0,$ since by commutator relations \eqref{conformalcommutator} in the beginning of this section $\Box (R-i\alpha) u=0,$ Proposition~\ref{proposition:conormal} implies that for $t>0$ and for all $\ep>0,$
$$
(R-i\alpha) u\in I^{m-(n-1)/4-1+\ep}(\RR\times X; N^* \{t-r=0\}),
$$
with symbol seminorms depending on those of $b.$ By definition of conormal distribution, this implies 
$$
(R-i\alpha)u\in  L^2(\RR_t;\D_{-m+1/2-0})
$$
for $t>0$, Thus by Lemma \ref{lemma:phg} we have the corresponding symbol decomposition of $a$ and we therefore proved the first part of our proposition.

To prove the quantitative estimates of the second part, we note that by
\eqref{precisetesting} and the quantitative propagation of conormality
from Proposition~\ref{proposition:conormal} we have for $t>0$
$$
b\equiv \left(-\lambda D_\lambda-im +\frac{r}{2\lambda} \Lap \right) a \in S^{m-1+\epsilon}(X; \RR_{\lambda})
$$
for all $\epsilon>0,$ together with symbol estimates: whenever
$\alpha_-<\beta_-<0<\alpha_+<\beta_+,$ for all $M,$
\begin{equation}
 \begin{split}
||b_{(\alpha_+,\beta_+)}||_{S^{m-1+\ep}_M} 
  & \leq C ||b_{(\alpha_-,\beta_-)}||_{S^{m-1}_{M'}} \\
  & = C ||(-\lambda D_\lambda-im +\frac{r}{2\lambda} \Lap )a_{(\alpha_-,\beta_-)}||_{S^{m-1}_{M'}} \\
  &\leq C ||a_{(\alpha_-,\beta_-)}||_{S^m_{M'}}\\
  \end{split}
\end{equation}
for $M'.$

Now without loss of generality we can assume $m=0$, otherwise we use $$\ta=\lambda^{-m} a,\quad \tb=\lambda^{-m} b$$ to make the order of $a$ zero. This yields
$$
b= \left(-\lambda D_\lambda +\frac{r}{2\lambda} \Lap \right) a \in S^{-1+\epsilon}(X; \RR_{\lambda}) \text{ for } t\in (\alpha_+,\beta_+)
$$
again enjoying the same type of effective estimates as $b.$
Thus
\begin{equation}
\label{a_dl}
D_\lambda a=-\lambda^{-1} b+\frac{r}{2\lambda^2} \Lap a \in S^{-2-\epsilon}(X; \RR_{\lambda}) \text{ for } t\in(\alpha_+,\beta_+),
\end{equation}
again the RHS enjoys effective estimates since $r/(2\lambda^2)
\Lap$ is a continuous map from symbols of order $s$ to symbols of order $s-2,$ and multiplication by powers of $\lambda$ is a continuous symbol map. In particular, then  
\begin{equation}
\label{est1}
 ||\lambda^{-1}b_{(\alpha_+,\beta_+)}||_{S^{m-2+\ep}_M} \leq C ||\lambda^{-1} a_{(\alpha_-,\beta_-)}||_{S^{m-1}_M} 
\end{equation}
 and 
\begin{equation}
\label{est2}
 ||\frac{r}{2\lambda^2} a_{(\alpha_+,\beta_+)}||_{S^{m-2+\ep}_M} \leq C  ||\lambda^{-2} a_{(\alpha_-,\beta_-)}||_{S^{m-2}_M} 
\end{equation}
Integrating \eqref{a_dl} from $\lambda=\pm\infty$ with
$C(r,\theta,\sgn\lambda)$ as constant of integration (cf.\ \cite[Propasition 2.1]{joshi1997intrinsic} for this strategy)
yields
\begin{align*}
a = &C(r,\theta,\sgn(\lambda)) + \int -\lambda^{-1} b+\frac{r}{2\lambda^2} \Lap \ta \, d\lambda\\
\equiv &C(r,\theta,\sgn(\lambda))+e(r,\theta,\sgn\lambda).
\end{align*}
Here the term $C(r,\theta,\sgn(\lambda))$ corresponds to the homogeneous term
in $a,$ while the integral term $e(r,\theta,\lambda)$ is a remainder
that lies in $S^{-1+\epsilon}$, which corresponds to $r_{m-1}$, and satisfies
effective estimates directly following from the above two estimates \eqref{est1} and \eqref{est2}. 
\end{proof}


We finally state the propagation of one-step polyhomogeneity as a corollary of our previous results, which we summarized as Theorem \ref{thm1.2} in the introduction: 

\begin{corollary}[Propagation of One-Step Polyhomogeneity]
\label{propa_phg}
Suppose that $\Box u=0$ and that for $t<0,$ $u \in I_{\phg}^{m-(n-1)/4}(\RR\times X; N^* \{t+r=0\}).$ Then for $t>0,$ $u \in I_{\phg}^{m-(n-1)/4}(\RR\times X; N^* \{t-r=0\})$ and has an oscillatory integral representation of the form \eqref{oscillatoryintegral} where the symbol $a$ is polyhomogeneous. 
\end{corollary}

\begin{proof}
We know 
$$\prod_{j=1}^k(R-i\alpha_j) u \in  I^{m-k-(n-1)/4}(\RR\times X; N^* \{t+r=0\}) $$with $\alpha_j=\frac{n+1}{2}+(m-j+1)$ as in Lemma \ref{lemma:fullphg} for $t<0$ and all $k\in\mathbf{N}$. By Proposition \ref{proposition:conormal} and the commutator relations \eqref{conformalcommutator},  
$$\prod_{j=1}^k(R-i\alpha_j) u \in  I^{m-k-(n-1)/4+0}(\RR\times X; N^* \{t-r=0\}) $$ for $t>0$ and all $k\in\mathbf{N}$. Thus by Lemma \ref{lemma:fullphg} and the fact that conormal distribution $I^{m-k-(n-1)/4+0}(\RR\times X; N^* \{t-r=0\})$ corresponds to iterative regularity $I\D_{-m-1/2+k-0}$, we conclude $u$ is a polyhomogeneous distribution. 
\end{proof}

\begin{remark}
This corollary in particular shows that the diffractive wave enjoys one-step polyhomegeneity, which improves the result of half-step polyhomogeneity given implicitly in Cheeger-Taylor \cite{cheeger1982diffraction} \cite[Theorem 5.1, 5.3]{cheeger1982diffraction2}. This half-step polyhomogeneity is further explicitly pointed out by Ford and Wunsch in \cite[Proof of Proposition 2.1]{ford2017diffractive}.
\end{remark}

\section{Diffraction Coefficient on Product Cones}

Assume $t>r'$. The half wave propagator can be decomposed as 
$$U(t)=U_G(t)+U_D(t),$$
where the first part is the geometric wave propagator and the second is the diffractive wave propagator with 
$$\text{WF}\ U(t)\backslash \text{WF}\ U_D(t) \subset \left\{ \text{bicharacteristics of the geometric geodesics} \right\}.$$
On the diffracted front $\mathscr{D}$ and away from the geometric front $\mathscr{G}$ (see Figure \ref{pic2}), we can write the kernel of the diffractive half wave propagator as 
\begin{equation}
\label{WK}
U_D(t)= (2\pi)^{-\frac{n+1}{2}}\int e^{i(r+r'-t)\cdot\lambda} K(r,\theta; r',\theta';\lambda)d\lambda |r^{n-1}drd\theta r'^{n-1}dr'd\theta'|^{\frac{1}{2}}
\end{equation}
where $K(r,\theta; r',\theta';\lambda)$ is a polyhomogeneous symbol of order $0$, and $U_D(t)$ is a conormal distribution to $\{r+r'=t\}$. Here we confuse the propagator with its Schwartz kernel and the assumption $t>r'$ is to ensure the existence of the diffractive front. This expression is due to Cheeger-Taylor\cite{cheeger1982diffraction}\cite{cheeger1982diffraction2} and Melrose-Wunsch \cite{Melrose-Wunsch1}, and it can be seen as a consequence of conormality of the diffractive wave.  We show later in this section that the symbol of the diffractive half wave kernel has the form: 
\begin{equation}
\label{diff_symbol}
 K(r,\theta; r',\theta';\lambda)\equiv\sum_j \sigma^T_j(t,r,r';\lambda)\varphi_j(\theta)\varphi_j(\theta')\ \mod S^{-\infty}
\end{equation}
where $\varphi_j$ is the $j$-th Fourier mode of $\Lap_{\theta}$ on $N$ and $\sigma^T_j(t,r,r';\lambda)$ is the total symbol of $j$-th mode of diffractive fundamental solution $E_j(t,r,r')$. We define  $\sigma^P_j(t,r,r';\lambda)$ to be the principal symbol of $\sigma^T_j$. And the \emph{diffraction coefficient} is defined to be 
$$K_0(r,\theta; r',\theta'):=\sum_j\sigma^P_j(t,r,r';\lambda)\varphi_j(\theta)\varphi_j(\theta').$$

The idea of this section is the following. We first consider diffractions of spherical waves. By \cite[Theorem 4.8]{Melrose-Wunsch1}, the spherical diffractive wave is cornormal to $\{r+r'=t\}$ for $t>r'$. Then we consider a mode-by-mode decomposition of the diffractive fundamental solution, and shows that the principal symbol of the diffractive half wave kernel is given by the sum of the principal symbol of the diffractive wave of each Fourier mode. This reduces the computation of the diffraction coefficient to each mode.

The construction is based on the functional calculus \cite{taylor2013partial} on product cones.  We first consider the exact solution (on a single mode) to the half wave equation. Recall the Laplacian on a product cone is
$$\Delta=D_r^2-i\frac{n-1}{r}D_r+\frac{1}{r^2}\Delta_{\theta}$$ and we define $\nu_j:=\sqrt{\mu_j^2+\alpha^2}$ with $\alpha=-\frac{n-2}{2}$, where $\mu_j^2, \varphi_j$ denote eigenvalues and eigenfunctions of $\Delta_{\theta}$. If we take 
$$g_j(r)=r^{-\frac{n-2}{2}}J_{\nu_j}(\lambda r),$$
then 
$$\Delta(g_j\varphi_j)=\lambda^2(g_j\varphi_j).$$
This can be seen by reducing $(\Lap-\lambda^2)(g_j\varphi_j)=0$ to a Bessel equation by change of variables. For a detailed treatment on this solution on a single Fourier mode, we refer to \cite{baskin2019scattering}. 
By the functional calculus on product cones,
\begin{equation}
\label{Function_Calc1}
f(\Delta) g(r,\theta)=r^{\alpha}\sum_j\left(\int_0^{\infty}f(\lambda^2)J_{\nu_j}(\lambda r)\lambda \Big(\int_0^{\infty}s^{1-\alpha} J_{\nu_j}(\lambda s)g_j(s)ds\Big)d\lambda\right)\varphi_j(\theta)
\end{equation}
for $g(r,\theta)=\sum_j g_j(r) \varphi_j(\theta)\in L^2(C(N); \mathbb{C})$. We define the operator $\nu$ on $N$ by
\begin{equation}
\nonumber
\nu=(\Lap_{\theta}+\alpha^2)^{\frac{1}{2}};
\end{equation}
the kernel of $f(\Lap)$ is thus a function on $\RR_+\times\RR_+$ taking values in operators on $N$, by the formula 
\begin{equation}
\label{Function_Calc2}
f(\Lap)= (rr')^{\alpha}\int^{\infty}_0 f(\lambda^2)J_{\nu}(\lambda r) J_{\nu}(\lambda r')\lambda d\lambda.
\end{equation}
Using \eqref{Function_Calc1}, we take $g(r,x)$ to be one single mode of spherical wave $\delta(r-r')\varphi_j(\theta)$ and the operator as half wave operator $U(t)=e^{-it\sqrt{\Delta}}$. Then 
$$
U(t)\big(\delta(r-r')\varphi_j\big)=r^{\alpha} \left( \int e^{-it\lambda}J_{\nu_j}(\lambda r) H(\lambda) \lambda \Big( H(r') {r'}^{1-\alpha} J_{\nu_j}(\lambda r') \Big)d\lambda \right)\varphi_j(\theta)
$$
for fixed $r'>0$, where $H$ is the Heaviside function. 

We define $\tilde{\chi}_{\nu_j}(\lambda):=\rho(\lambda) \lambda \Big(  \tilde{\rho}(r') {r'}^{1-\alpha} J_{\nu_j}(\lambda r') \Big)$ with  $\rho(\lambda)$ and $\tilde{\rho}(\lambda)$ being smooth cutoffs away from $0$ and equal to $1$ for $\lambda>1$. Thus the solution should have the form 
$$
u_{\nu_j}(t,r,r',\theta)\equiv r^{\alpha} \left( \int e^{-it\lambda}J_{\nu_j}(\lambda r)\tilde{\chi}_{\nu_j}(\lambda) d\lambda \right)\varphi_j(\theta) \mod\CI
$$
for fixed $r'>0$. Note that the solution also has singularities at $\{r=r'+t\}$ apart from the diffractive singularities. 

From the previous discussion we have
\begin{equation}
\nonumber
U(t)\big(\delta(r-r') \varphi_j(\theta)\big)\equiv \int E_D(t,r,r',\theta,\theta')\varphi_j(\theta')d\theta' =E_{j}(t,r,r') \varphi_j(\theta),
\end{equation}
modulo the singularities at $\{r=r'+t\}$ for the first equation. Now we compute the diffractive fundamental solution $E_D(t,r,r',\theta, \theta')$. We first regularize it by averaging it angularly to instead study 
$$
u_\varphi(t,r,r',\theta)=\int E_D(t,r,\theta,r',\theta') \varphi(\theta') \, d\theta'
$$
for an arbitrary $\varphi \in \mathcal{C}^{\infty}_c(N)$ supported close to a single point $\theta_0\in N.$ Using the Plancherel theorem, Fourier expanding $\varphi$ in $N$,  i.e., taking the eigenfunction expansion gives
$$u_\varphi=\sum_j c_{\nu_j}  E_{j}(t,r,r') \varphi_j(\theta)$$ with $c_{\nu_j}=\langle \varphi, \varphi_j\rangle_{L^2}$ the corresponding Fourier coefficient. 


Now what we can compute by the asymptotic expansion of Bessel functions is the principal symbol of the conormal solution $u_{\nu_j}$ at $N^*\{r+r'=t\}$ (as we will do in the later part of this section); for now we write these principal symbols 
$\sigma_{j}^P(t,r,r',\lambda)\varphi_j(\theta).$ Formally, the principal symbol of $u_{\varphi}$ is the sum of the principal symbols of $u_{\nu_j}$, though we have to be careful to show that the subprincipal symbols of $u_{\nu_j}$ will not add up and contribute to the principal symbol of $u_{\varphi}$. This leads to the following theorem: 

\begin{theorem}[Principal Symbols of Diffraction]
\label{Diffraction_Symbol}
The principal symbol $\sigma^P$ of $u_\varphi$ is equal to the sum of principal symbols of mode-by-mode solutions:
\begin{equation}
\sigma^P(u_\varphi) = \sum_j c_{\nu_j} \sigma^P_{j}(t,r,r',\lambda) \varphi_j(\theta),
\end{equation}
i.e., the subprincipal symbols of $u_{\nu_j}$ won't add up and contribute to principal symbol of $u_{\varphi}$. 
\end{theorem}

\begin{proof}
The convergence of this sum is due to the fact that $c_{\nu_j}$ are rapidly decaying with respect to $\nu_j$ as $j\rightarrow\infty$, which comes from the fact that $\varphi$ is $\mathcal{C}^{\infty}$ so the Fourier coefficient rapidly decays; we can check that
the series of principal symbols converges.  However, we need information about the growth rate of the symbol \emph{remainders}: the equality of conormal distributions tells us that really 
$$\sigma^T(u_\varphi) = \sum_j c_{\nu_j} \sigma^T_{j}(t,r,r',\lambda) \varphi_j(\theta),$$
where we use $\sigma^T$ to denote the total symbol, i.e., the full amplitude of the conormal distribution with the canonical choice of phase function $\phi(t,r,\lambda)=(t\pm r)\lambda.$ Thus, it remains to check that the sum of remainder terms
\begin{equation}\label{subprincipalseries}
 \sum_j c_{\nu_j} (\sigma^T_j-\sigma^P_{j})\varphi_j(\theta)
\end{equation}
converges in the topology of symbols of order $m-1+\epsilon.$  By Proposition~\ref{proposition:phg} we find that any desired symbol semi-norm of $\sigma^T_j-\sigma_{m,j}$ is bounded by some symbol semi-norm of the symbol of solution $u_{\nu_j}$ for $t\ll 0.$ Examination of the initial data shows that each of these norms grows at most \emph{polynomially} in $\nu_j$ (with the growth arising from $\theta$ derivatives). Thus, since $c_{\nu_j}$ decays rapidly, the series \eqref{subprincipalseries} does indeed converge in every symbol semi-norm with respect to the $S^{m-1+\ep}$ topology, and the subprincipal terms cannot affect the principal symbol of the sum.\footnote{It certainly \emph{can} happen that lower-order terms in a sum affect the principal symbol of the result: consider $\delta=\sum_k e^{ik\theta}$ on the circle, and regard the RHS as a sum of conormal distributions with symbols of order $0$, which happen to have vanishing principal symbol. Of course the problem is that the symbol remainder terms grow nastily in $k$. }
\end{proof}

Following Theorem \ref{Diffraction_Symbol}, we now construct the principal symbol of the diffractive fundamental solution. We fix $\chi\in \CI(N)$ equal to $1$ near $\theta_0$ and use the above results for all $\varphi$ supported on $\{\chi=1\}.$
We have established that $$\sigma^P(u_{\varphi})=\sum_j c_{\nu_j} \sigma_j^P(t,r,r',\lambda) \varphi_j(\theta).$$  
Now let $\varphi$ approach $\delta(\theta')$ in the sense of distribution (with $\theta'$ not geometrically related to $\theta$), so that its Fourier coefficients $c_{\nu_j}$ approach $\varphi_j(\theta').$ We then obtain in the limit, in a neighborhood of any pair $\theta$ and $\theta'$ that are related by strictly diffractively geodesics, i.e., for $\theta, \theta'$ with $d_{h}(\theta, \theta')\neq\pi$, 
$$
\sigma^P(E_D(t,r,r',\theta,\theta'))  = \sum \sigma^P_{j}(t,r,r', \lambda)\varphi_j (\theta)\varphi_j(\theta'),
$$
as desired. 

In order to get the diffraction coefficient, now it remains to compute the principal symbol of the diffactive fundamental solution for each mode: $\sigma^P_{j}(t,r,r', \lambda)$. We employ the functional calculus on product cones and the conormality of diffractive waves. 

We consider the kernel of the half wave propagator given by \eqref{Function_Calc2}. Again here we use the Hankel asymptotics to get the diffractive coefficient. Consider \eqref{Function_Calc2} acting on a single mode $\varphi_j(\theta)$:
\begin{equation}
\label{wave_mode}
E_j(t,r,r')\varphi_j(\theta)\equiv (rr')^{\alpha}\Bigg(\int^{\infty}_0 e^{-it\lambda}J_{\nu_j}(\lambda r) J_{\nu_j}(\lambda r')\lambda d\lambda \Bigg) \varphi_j(\theta) 
\end{equation} 
modulo the singularities of the right hand side at $N^*\{r=r'+t\}$.

Note that for positive $\nu$ and $z$, the Bessel function $J_{\nu}(z)$ is the real part of Hankel function $H^{(1)}_{\nu}(z)$. Thus using the asymptotic formulas of Hankel functions from \cite[10.17.5]{NIST:DLMF}, we can extract the leading part of $J_{\nu_j}(\lambda r)$ as the principal symbol with phase variable $\lambda$:
\begin{equation}
\label{Bessel_Asym}
\begin{split}
J_{\nu_j}(\lambda r)
& \equiv \left(\frac{1}{2\pi \lambda r}\right)^{\frac{1}{2}} \left(e^{i(\lambda r-\frac{\nu_j\pi}{2}-\frac{\pi}{4})}+ e^{-i(\lambda r-\frac{\nu_j\pi}{2}-\frac{\pi}{4})} \right)\ \mod S^{-\frac{3}{2}+0}
\end{split}
\end{equation}
We now combine this Bessel asymptotics together with \eqref{wave_mode} to get the diffractive principal symbol. Thus, 
\begin{equation}
\nonumber
\begin{split}
E_j(t,r,r')\varphi_j(\theta)
 &\equiv (rr')^{\alpha}\Bigg(\int \rho(\lambda) e^{-it\lambda}J_{\nu_j}(\lambda r) J_{\nu_j}(\lambda r')\lambda d\lambda \Bigg) \varphi_j(\theta)\ \mod \CI\\
 &\equiv \frac{1}{2\pi}(rr')^{\alpha-\frac{1}{2}}\Bigg(\int e^{i\lambda(r+r'-t)} e^{-i(\nu_j\pi+\frac{\pi}{2})} d\lambda\Bigg)\varphi_j(\theta) 
\end{split}
\end{equation} 
modulo singularities at the conormal bundle $N^*\{r=r'+t\}$ and the lower order singularities at the conormal bundle $N^*\{r+r'=t\}$. The second equality is due to the fact that diffractive wave is conormal to $\{r+r'=t\}$ \cite[Theorem 4.8]{Melrose-Wunsch1}, so the only part in $J_{\nu_j}(\lambda r)$ and $J_{\nu_j}(\lambda r')$ that contributes to the diffractive principal symbol is each of their first terms in the asymptotic expansion \eqref{Bessel_Asym}, and the remaining terms are smooth near $N^*\{r+r'=t\}$, hence will not contribute to the diffractive wave. Now, comparing the above equation with the general formula for the diffractive half wave kernel \eqref{WK}, we have the diffraction coefficient: 
\begin{equation}
\label{Diffraction_Coefficient}
K_0(r,\theta; r',\theta')= -\frac{i}{2\pi}(rr')^{-\frac{n-1}{2}} e^{-i\nu\pi},
\end{equation}
where $\nu=\sqrt{\Lap_{\theta}+\left(\frac{n-2}{2}\right)^2}$.

\section{Scattering Matrix on Product Cones}
Consider the leading order behaviors of the solutions of
$$\left(\Lap-\lambda^2\right)u=0$$
under the asymptotic condition:
\begin{equation}
\label{u_asym}
u\sim a_+(\theta)r^{-\frac{n-1}{2}}e^{i\lambda r}+a_-(\theta)r^{-\frac{n-1}{2}}e^{-i\lambda r}+\mathcal{O}(r^{-\frac{n+1}{2}})\ \text{ as }r\rightarrow\infty,
\end{equation}
where $a_-/a_+$ is called the incoming/outgoing coefficient. Then $a_+(\theta)$ is uniquely determined by $a_-(\theta)$ and the scattering matrix $S(\lambda)$ is the unitary map from $a_-(\theta)$ to $a_+(\theta)$ for $\lambda\in\RR\backslash\{0\}$. This property is known for smooth asymptotically Euclidean manifolds \cite{melrose1994spectral}, and we show below in Proposition \ref{SM_def} it is also true on product cones. Meanwhile, we show that the scattering matrix on a product cone is
\begin{equation}
\nonumber
S(\lambda)= -i e^{-i\pi\nu}
\end{equation}
where $\nu=\sqrt{\Lap_{\theta}+\left(\frac{n-2}{2}\right)^2}$, which is related to the diffraction coefficient \eqref{Diffraction_Coefficient} as
\begin{equation}
\nonumber
S(\lambda)= 2\pi (rr')^{\frac{n-1}{2}} K_0(r,\theta;r',\theta'). 
\end{equation}
Here we should note that we only consider the smooth part of the scattering matrix. By \cite{melrose1996scattering}, on smooth asymptotically Euclidean manifolds (smooth manifolds with large conical ends), the scattering matrix is a Fourier integral operator with the canonical relation given by geodesic flow at time $\pi$. In the previous section, we found the diffraction coefficient of the points on product cones which are strictly diffractively related. This corresponds to the smooth part of the scattering matrix, i.e. $S(\lambda,\theta,\theta')$ for $d_h(\theta,\theta')\neq\pi$, where $S(\lambda,\theta,\theta')$ is the kernel of the scattering matrix. From now on we use the name \emph{scattering matrix} without saying that it means the smooth part. 

We define the scattering matrix through the following proposition: 
\begin{proposition}[The Scattering Matrix on Product Cones]
\label{SM_def}
The \emph{scattering matrix} $S(\lambda)$ on the product cone $C(N)$ for $\lambda\in\RR\backslash\{0\}$ is a unitary operator: 
\begin{equation}
\nonumber
\begin{split}
S(\lambda): \CI(N)&\longrightarrow \CI(N) \\
           a_{-}(\theta) &\mapsto a_{+}(\theta)
\end{split}
\end{equation}
where $a_+(\theta)$ is the \emph{outgoing} coefficient in the asymptotic expansion \eqref{u_asym} and it is uniquely determined by the \emph{incoming} coefficient $a_-(\theta)$. Moreover, the scattering matrix on a product cone takes the form:
\begin{equation}
\nonumber
    S(\lambda)= -i e^{-i\pi\nu}
\end{equation}
where $\nu=\sqrt{\Lap_{\theta}+\left(\frac{n-2}{2}\right)^2}$. 
\end{proposition}

\begin{proof}
Consider the homogeneous equation on $C(N)$
\begin{equation}
\label{homo_helm}
\left(\Delta - \lambda^2\right)u=0.
\end{equation}
By Section 2, this is
\begin{equation}
\nonumber
\left(\partial_r^2+\frac{n-1}{r}\partial_r-\frac{1}{r^2}\Delta_{\theta}+\lambda^2\right)u=0.
\end{equation}
Now we consider an eigenfunction decomposition of $u\in L^2(C(N))$ by eigenfunctions on $N$ of $\Delta_{\theta}$. i.e.,
\begin{equation*}
L^{2}(C(N); \mathbb{C}) = \bigoplus_{j=0}^{\infty}L^{2}(\RR_+; E_{j}), \quad u(r,\theta) = \sum_{j=0}^{\infty}v_{j}(r) \varphi_{j}(\theta),
\end{equation*}
where the first space is defined via the volume form induced by the conic metric, and the latter spaces can be identified with the space $L^2(\RR_+; r^{n-1}dr)$. $E_j$ denotes the $j$-th eigenspace.  

The equation \eqref{homo_helm} thus becomes
$$\sum_j \varphi_j(\theta) \cdot \left( \partial_r^2+\frac{n-1}{r}\partial_r+\lambda^2-\frac{\mu_j^2}{r^2}\right) v_j(r) =0,$$
where $\mu_j^2$ is the eigenvalue to $\Lap_{\theta}$ with the eigenfunction $\varphi_j$. We thus reduce the equation \eqref{homo_helm} to 
\begin{equation}
\label{red_Helm}
\left(\partial_r^2+\frac{n-1}{r}\partial_r+\left(\lambda^2-\frac{\mu_j^2}{r^2}\right)\right)v_j(r)=0
\end{equation}
for all $j$. By changing variable $\rho$ with $\rho=\lambda r$, we have
\begin{equation}
\nonumber
\left(\partial_{\rho}^2+\frac{n-1}{\rho}\partial_{\rho}+\left(1-\frac{\mu_j^2}{\rho^2}\right)\right)v_j(r)=0.
\end{equation}
Writing $v_j(r)=\rho^{\sigma}\omega_j(\rho)$, we can replace the previous equation by
\begin{equation}
\nonumber
\sigma(\sigma-1)\rho^{\sigma-2}\omega_j+2^{\sigma}\rho^{\sigma-1}\partial_{\rho}\omega_j+\rho^{\sigma}\partial^2_{\rho}\omega_j+\frac{n-1}{\rho}\left(\sigma\rho^{\sigma-1}\omega_j+\rho^{\sigma}\partial_{\rho}\omega_j\right)+\left(1-\frac{\mu_j^2}{\rho^2}\right)\rho^{\sigma}\omega_j=0.
\end{equation}
Rewrite it into the following form
\begin{equation}
\nonumber
\left(\partial^2_{\rho}+\frac{2\rho+(n-1)}{\rho}\partial_{\rho}+\frac{\sigma(\sigma-1)+(n-1)\sigma}{\rho^2}+\left(1-\frac{\mu_j^2}{\rho^2}\right)\right)\omega_j=0.
\end{equation}
Setting $2\sigma+(n-1)=1$, i.e., $\sigma=1-\frac{n}{2}$,
the equation above then becomes
\begin{equation}
\nonumber
\left(\partial^2_{\rho}+\frac{1}{\rho}\partial_{\rho}+\left(1-\frac{\mu_j^2+(1-n/2)^2}{\rho^2}\right)\right)\omega_j=0.
\end{equation}
This is a homogeneous Bessel equation
\begin{equation}
\label{homo_Bes}
\omega_j''+\frac{1}{\rho}\omega_j'+\left(1-\frac{\nu_j^2}{\rho^2}\right)\omega_j=0
\end{equation}
with $\nu_j=\sqrt{\mu_j^2+(1-n/2)^2}$. Its general solutions in terms of Bessel/Hankel functions are the linear combinations of 
$H^{(1)}_{\nu}(z)$ and $H_{\nu}^{(2)}(z)$. We can then construct solutions to \eqref{homo_helm} with the prescribed boundary condition using Bessel functions asymptotics as following. Consider the general solutions to \eqref{homo_Bes}
\begin{equation}
\nonumber
\omega_j=C_1H_{\nu_j}^{(1)}(\rho)+C_2H_{\nu_j}^{(2)}(\rho).
\end{equation}
Noticing that from the above change of variables, $v_j=\rho^{1-\frac{n}{2}}\omega_j(\rho)$ and $\rho=\lambda r$, we can obtain general solutions to \eqref{red_Helm}:
\begin{equation}
\nonumber
v_j=C_{1}\cdot(\lambda r)^{1-n/2}H_{\nu_j}^{(1)}(\lambda r)+C_{2}\cdot(\lambda r)^{1-n/2}H_{\nu_j}^{(2)}(\lambda r)
\end{equation}
and thus they have general solutions to \eqref{homo_helm}:
\begin{equation}
\label{u}
u=\sum_j \left(C_{1,j}\cdot(\lambda r)^{1-n/2}H_{\nu_j}^{(1)}(\lambda r)+C_{2,j}\cdot(\lambda r)^{1-n/2}H_{\nu_j}^{(2)}(\lambda r)\right) \varphi_j(\theta).
\end{equation}
By the asymptotic behaviors of Hankel functions, we notice that the solutions thus have the asymptotic expansion:
\begin{equation}
\label{asymp_helm}
u\sim a_+(\theta)r^{-\frac{n-1}{2}}e^{i\lambda r}+a_-(\theta)r^{-\frac{n-1}{2}}e^{-i\lambda r}+\mathcal{O}(r^{-\frac{n+1}{2}})\ \text{ as }r\rightarrow\infty,
\end{equation}
where the first term is called outgoing and the second incoming. 

As in the construction of the resolvent on the product cone \cite[Theorem 2.1]{baskin2019scattering}, we consider the boundary behaviors of Bessel functions. The choice of the Friedrichs extension requires that both $v_j$ and $v'_j$ lie in the weighted $L^2$ space near $0$. By the asymptotic formula of Bessel functions $J_{\nu}(z)$ and $Y_{\nu}(z)$ from \cite[10.7.3, 10.7.4]{NIST:DLMF}, the existence of $Y_{\nu}(z)$ fails the Friedrichs extension condition. And since
$$H_{\nu}^{(1)}(z)=J_{\nu}(z)+iY_{\nu}(z)\ \text{and}\  H_{\nu}^{(2)}(z)=J_{\nu}(z)-iY_{\nu}(z)$$
for $\nu>0$, the facts above and in particular the Friedrichs extension imply that in \eqref{u} the coefficients $C_{1,j}=C_{2,j}=C_j$. Thus $v_j$ must be a multiple of $r^{-(n-2)/2}J_{\nu_j}(\lambda r)$ near $r=0$. And the solution then becomes:
\begin{equation}
\label{u_1}
u=\sum_j 2C_j\cdot(\lambda r)^{1-n/2}J_{\nu_j}(\lambda r) \varphi_j(\theta).
\end{equation}
Again, from \cite[10.17.3]{NIST:DLMF} by asymptotics of $J_{\nu}(z)$:
\[J_{\nu}\left(z\right)\sim\left(\frac{2}{\pi z}\right)^{\frac{1}{2}}\*\left( \cos\omega\sum_{k=0}^{\infty}(-1)^{k}\frac{a_{2k}(\nu)}{z^{2k}}-\sin\omega\sum_{k=0}^{\infty}(-1)^{k}\frac{a_{2k+1}(\nu)}{z^{2k+1}}\right),\text{ as } z\rightarrow0\]
where $\omega=z-\tfrac{1}{2}\nu\pi-\tfrac{1}{4}\pi$ and $a_k(\nu)$ constants of $\nu$, we can extract the leading part of $u$ as 
$$u\sim \sum_j 2C_j\cdot\left(\lambda r\right)^{1-n/2}\left(\frac{1}{2\pi\lambda r}\right)^{\frac{1}{2}}\left(e^{-\frac{i\pi}{4}-\frac{i}{2}\pi\nu_j}e^{i\lambda r}+e^{\frac{i\pi}{4}+\frac{i}{2}\pi\nu_j}e^{-i\lambda r}\right)\varphi_j(\theta).$$
For each mode $\varphi_j(\theta)$, the scattering matrix $S(\lambda)$ acts like
$$a_{-,j}(\theta):=e^{\frac{i\pi}{4}+\frac{i}{2}\pi\nu_j}\varphi_j\mapsto\ a_{+,j}(\theta):=e^{-\frac{i\pi}{4}-\frac{i}{2}\pi\nu_j}\varphi_j.$$
Thus the scattering matrix acts diagonally on modes and takes the form: 
$$S(\lambda)=-ie^{-i\pi\nu}$$
where $\nu=\sqrt{\Lap_{\theta}+\left(\frac{n-2}{2}\right)^2}$. The unitarity follows directly from this expression.

Now we show the uniqueness of the scattering matrix. It suffices to show the uniqueness for each mode. For any $g(\theta)\in\CI(N)$, if there are two solutions $u_1$ and $u_2$ to
\begin{equation}
\label{homo_Bessel}
    \left(\Lap-\lambda^2\right)u=0
\end{equation}
with $g(\theta)$ the incoming boundary condition, then $u=u_1-u_2$ is an outgoing solution to the homogeneous stationary wave equation \eqref{homo_Bessel}. Thus, each mode component of $u$ needs to satisfy the outgoing condition and it takes the form $$H^{(1)}_{\nu_j}(\lambda r)=J_{\nu_j}(\lambda r)+iY_{\nu_j}(\lambda r),$$
since $H^{(1)}_{\nu_j}(\lambda r)$ gives the outgoing part in the asymptotic expansions \eqref{u_asym}. 
Then each mode component of $u$ has to be zero since $Y_{\nu_j}(\lambda r)$ does not lie in $L^2$ at $r=0$, otherwise it contradicts to the requirement of the Friedrichs extension. Hence $a_+(\theta)$ must be uniquely determined by $g=a_-(\theta)$.
\end{proof}

Combing the expression of the scattering matrix with the conclusion of the diffraction coefficient in the previous section, we have proved Theorem \ref{thm1.1} which we restate more precisely as follows: 
\begin{theorem*}
Away from the intersection of geometric wave front and diffractive front, i.e., for $d_h(\theta,\theta')\neq\pi$, the diffraction coefficient, which is the principal symbol of the diffractive half wave kernel,  $K_0(r, \theta; r', \theta')$ and the kernel of the scattering matrix $S(\lambda)$ satisfy the following relation:
\begin{equation}
\nonumber
K_0(r, \theta, r', \theta')= (2\pi)^{-1}(rr')^{-\frac{n-1}{2}}S(\lambda, \theta,\theta'), 
\end{equation}
where $S(\lambda, \theta,\theta')$ is the kernel of the scattering matrix. 
\end{theorem*}

\section{From the Radiation Field to the Scattering Matrix}

We now proceed to study the diffraction part of the radiation field of the wave equation on product cones. By using the notion of the radiation field, we can give an alternative proof of the Theorem \ref{thm1.1}. 

We use some ideas developed by Friedlander \cite{friedlander2001notes} and S{\'a} Barreto \cite{sa2003radiation} on the radiation field of asymptotically Euclidean manifolds. The radiation field on product cones is also studied by Baskin and Marzuola\cite{baskin2019radiation}.\footnote{In \cite{baskin2019radiation}, they studied the radiation field of the solution to the wave equation away from the cone point, while we only use the radiation field corresponding to the diffractive fundamental solution.} We define the radiation field through a theorem proved by Friedlander in \cite{friedlander2001notes} for smooth manifolds, and in the setting of product cones $X=C(N)$ by Baskin and Marzuola \cite{baskin2019radiation}: 

\begin{theorem}(The Radiation field)
\label{rf}
Let $f_0, f_1\in \CI_0(X)$ be smooth functions with compact support in X. If $u(t,r,\theta)\in\CI(\RR_+\times X)$ solves the wave equation with Cauchy data: 
\begin{equation}
\nonumber
\begin{cases}
& \Box u(t,r,\theta)=0 \ \text{on}\ \RR\times X\\ 
& u(0,r,\theta)=f_0,\ D_tu(0,r,\theta)=f_1,
\end{cases}
\end{equation}
where $(r,\theta)\in\RR_+\times N$, then there exist $w_k\in\CI(\RR\times N)$, such that 
\begin{equation}
\nonumber
r^{\frac{n-1}{2}}(Hu)(s+r,r,\theta)\sim\sum_{k=0}^{\infty}r^{-k}w_k(s,\theta), \text{as } r\rightarrow \infty
\end{equation}
where $H(t)$ denotes the Heaviside function. In particular, 
\begin{equation}
\nonumber
r^{\frac{n-1}{2}}(Hu)(s+r,r,\theta)\rvert_{r\rightarrow\infty}=w_0(s,\theta)
\end{equation}
is well defined, and it is called the \emph{radiation field} of $u$ as in \cite[Proposition 2]{friedlander2001notes}. 
\end{theorem}

\begin{remark}
 Friedlander showed the existence of the radiation field in the context of the smooth scattering manifold. In our case there is a singularity at the cone point. This is not an issue here, since we are considering the radiation field away from the cone point.
\end{remark}
For solutions to the wave equation, we have the energy norm $\|\cdot\|_E$: 
\begin{equation}
\nonumber
\|u\|_E=\frac{1}{2}\int_X \left(|du|_g^2+|u_t|^2\right)dg
\end{equation}
and define the wave group $W(t)$ by
\begin{equation}
\nonumber
\begin{split}
W(t)&: \CI_0(X^{\mathrm{o}})\times \CI_0(X^{\mathrm{o}})\longrightarrow \CI_0(X^{\mathrm{o}})\times \CI_0(X^{\mathrm{o}})\\
      &W(t)(f_0,f_1)=(u, u_t),\ t\in \RR.
\end{split}
\end{equation}
We know by conservation of energy that $W(t)$ is a strongly continuous group of unitary operators. 

We now define a map
\begin{equation}
\nonumber
\begin{gathered}
\rp: \CI_0(X^{\mathrm{o}})\times \CI_0(X^{\mathrm{o}})\longrightarrow \CI(\RR\times N)\\
\rp(f_0,f_1)(s,\theta)=r^{\frac{n-1}{2}}(D_tHu)(s+r,r,\theta)\rvert_{r\rightarrow\infty}=:D_sw^+_0(s,\theta)\big(=D_sw_0(s,\theta)\big)
\end{gathered}
\end{equation}
which is called the \emph{forward radiation field}. Its existence follows from Theorem \ref{rf}. Similarly we can define the \emph{backward radiation field} as 
\begin{equation}
\nonumber
\begin{gathered}
\rn: \CI_0(X^{\mathrm{o}})\times \CI_0(X^{\mathrm{o}})\longrightarrow \CI(\RR\times N)\\
\rn(f_0,f_1)(s,\theta)=r^{\frac{n-1}{2}}(D_tH_-u)(s-r,r,\theta)\rvert_{r\rightarrow\infty}=D_sw^-_0(s,\theta),
\end{gathered}
\end{equation}
where $H_-(t)=H(-t)$.

S{\'a} Barreto also proved in \cite{sa2003radiation} that the forward and backward radiation fields are in fact unitary on smooth asymptotically Euclidean manifolds under the energy norm of the Cauchy data. It leads to the definition of the scattering operator which is essentially the Fourier conjugation of the scattering matrix. 

\begin{theorem}
The maps $\rpn$ extend to isometries
\begin{equation}
\nonumber
\rpn: H_{E}(X)\longrightarrow L^2(\RR\times N).
\end{equation}
The scattering operator defined by 
\begin{equation}
\label{ScOp}
\mathscr{S}=\rp\circ\rn^{-1}
\end{equation}
is unitary on $L^2(\RR\times N)$; the \emph{scattering matrix} $S$ is given by conjugating the scattering operator with the partial Fourier transform in the $s$-variable: 
\begin{equation}
\label{SMSO}
S=\mathcal{F}\mathscr{S}\mathcal{F}^{-1}.
\end{equation}
\end{theorem}
Here we note that the proof of this theorem in \cite{sa2003radiation} can be extended to product cones from smooth asymptotically Euclidean manifolds. This is because the proof relies on the fact that the Laplacian on product cones has purely absolute continuous spectrum so that we could apply the proposition that the Poisson operators give isometries from the absolute continuous spectral subspace of $\Lap$ to $L^2(\RR \times N)$ \cite[Proposition 9.1]{hassell1999spectral} to give an isometry between the energy norm of the initial data $(f_0,f_1)$ and the Fourier transform of the forward/backward radiation field as in \cite{sa2003radiation}. Otherwise, although there are still isometries between the absolute continuous spectral subspace of $\Lap$ and $L^2(\RR \times N)$, the energy norm would need to include the contribution from the discrete eigenmodes.

In the rest of this section, we construct the scattering matrix from the diffractive coefficient using the radiation field, and we shall see that the scattering matrix agrees with the diffraction coefficient up to scaling. In other words, using the radiation field, we give a second proof to Theorem \ref{thm1.1}. 

Assume $z=(r,\theta)$ and $z'=(r',\theta')$ are points on the cone $C(N)$. We first consider the forward fundamental solution $E(z,z',t)$ solving
$$\Box_{(r,\theta,t)}E=\delta(z-z')\delta(t),\ \ \text{supp} E\subset\{t\geq d_{C(N)}\left(z,z'\right)\}.$$
For fixed $z'\in C(N)$, the fundamental solution has a radiation field 
\begin{equation}
\nonumber
E_{\infty}(s,\theta,z')=\lim_{r\rightarrow\infty} r^{\frac{n-1}{2}} E(s+r,r,\theta,z'). 
\end{equation} 
Friedlander in \cite{friedlander2001notes} also points out that we can get the inverse of the radiation field through the following formula:
$$u(r,\theta,t)=-2\int E_{\infty}(s'-t,\theta',r,\theta)\partial_sw(s',\theta')|h(0,\theta')|^{1/2}|ds'd\theta'|.$$
Then from \eqref{ScOp}, we get the kernel of the scattering operator $\mathscr{S}$ in terms of the fundamental solution: 
\begin{equation}
\label{kerrf}
\begin{split}
K_{\mathscr{S}}&=\lim_{r\rightarrow\infty}r^{\frac{n-1}{2}}\partial_tE_{\infty}(s+s'+r,\theta',r,\theta)\\
                         &=\lim_{\substack{r\rightarrow\infty\\r'\rightarrow\infty}} 2(rr')^{\frac{n-1}{2}}\partial_tE(s-s'+r+r',r,\theta,r',\theta').
\end{split}
\end{equation} 
In terms of the wave propagator, we know the fundamental solution is given by $$E(z,z',t)=\frac{\sin t\sqrt{\Lap}}{\sqrt{\Lap}}\delta(z-z').$$ 

Now we compute the scattering operator using \eqref{kerrf} and the kernel of the diffractive wave propagator. Since we showed the diffractive wave enjoys one-step polyhomogeneity in Section 3, by the formula of the kernel of the diffractive wave propagator, we can make a WKB style ansatz for the diffractive wave of  $E(z,z',t)$ as the following: 
\begin{equation}
\nonumber
E(z,z',t)\equiv (rr')^{-\frac{n-1}{2}} \int e^{i(r+r'-t)\lambda} a(r,r',\theta,\theta',\lambda)\, d\lambda 
\end{equation}
modulo singularities other than $N^*\{r+r'=t\}$ with symbol
$$a=\sum_{k=0}^{\infty}\bar{a}_{k}(r)\lambda^{-k},$$
where $\bar{a}_0$ is the diffraction coefficient. Here the symbol $a$ and each term $\bar{a}_k$ in its expansion depend on $(r,r',\theta,\theta')$, though we only emphasis the $r$ dependence for our purpose.

By the WKB type approximation, this will be reduced to the symbol expansion $a_k$ solving a series of transport equations. This leads to the result that the diffractive wave takes the form of polyhomogeneity in $\lambda r$: 
\begin{equation}
\nonumber
(rr')^{-\frac{n-1}{2}} \int e^{i(r+r'-t)\lambda} \left(\sum_{k=0}^{\infty}a_{k}\cdot(\lambda r)^{-k}\right)\, d\lambda 
\end{equation}
where $a_{k}$ does not depend on $r$ or $\lambda$. Now the kernel of the scattering operator of the diffractive wave can be computed as the following using equation \eqref{kerrf}:  
\begin{equation}
\nonumber
K_{\mathscr{S}}= \lim_{\substack{r\rightarrow\infty\\r'\rightarrow\infty}} 2(rr')^{\frac{n-1}{2}}\int_{\RR_{\lambda}}e^{i(s'-s)\lambda}(-i)\frac{\lambda}{2|\lambda|}\rho(\lambda)\frac{(rr')^{-\frac{n-1}{2}}}{2\pi}\left(H(\lambda)e^{-i\pi\nu}+H(-\lambda)e^{i\pi\nu}\right)d\lambda
\end{equation}
as the Schwartz kernel, where $\rho(\lambda)\in\CI(\RR)$ satisfies $\rho\equiv 1$ for $|\lambda|>2$ and $\rho\equiv 0$ for $|\lambda|<1$. The equation holds with only the leading part on $r,r'$ left, because the remainder terms have lower order in $r,r'$, and they become zero in the limit. In other words, the diffraction coefficient is the only part that contributes to this piece of the scattering operator. 

Taking the Fourier conjugation of the scattering operator computed above, by \eqref{SMSO} the scattering matrix is 
\begin{equation}
\nonumber
\begin{split}
S(\lambda)                       &=-i\frac{\lambda}{|\lambda|}\rho(\lambda)\left(H(\lambda)e^{-i\pi\nu}+H(-\lambda)e^{i\pi\nu}\right)\\                        
&=-i H(\lambda)e^{-i\pi\nu}+i H(-\lambda)e^{i\pi\nu}.
\end{split}
\end{equation}
For $\lambda\neq 0$, the scattering matrix takes the form: 
$$S(\lambda)=-ie^{-i\pi\nu}$$
where $\nu=\sqrt{\Lap_{\theta}+\left(\frac{n-2}{2}\right)^2}$, which agrees with what we obtained from the direct computation in Section 5. 

\bibliography{reference}

\begin{thebibliography}{MVW08}

\bibitem[BM19]{baskin2019radiation}
Dean Baskin and Jeremy~L Marzuola.
\newblock The radiation field on product cones.
\newblock {\em arXiv preprint arXiv:1903.02654}, 2019.

\bibitem[BW19]{wunsch2019phg}
Dean Baskin and Jared Wunsch.
\newblock Propagation of polyhomogeneity.
\newblock {\em unpublished}, 2019.

\bibitem[BY19]{baskin2019scattering}
Dean Baskin and Mengxuan Yang.
\newblock Scattering resonances on truncated cones.
\newblock {\em arXiv preprint arXiv:1903.02654, To appear}, 2019.

\bibitem[CT82a]{cheeger1982diffraction}
Jeff Cheeger and Michael Taylor.
\newblock On the diffraction of waves by conical singularities. i.
\newblock {\em Communications on Pure and Applied Mathematics}, 35(3):275--331,
  1982.

\bibitem[CT82b]{cheeger1982diffraction2}
Jeff Cheeger and Michael Taylor.
\newblock On the diffraction of waves by conical singularities. ii.
\newblock {\em Communications on Pure and Applied Mathematics}, 35(4):487--529,
  1982.

\bibitem[{\relax DLMF}]{NIST:DLMF}
{\it NIST Digital Library of Mathematical Functions}.
\newblock http://dlmf.nist.gov/, Release 1.0.23 of 2019-06-15.
\newblock F.~W.~J. Olver, A.~B. {Olde Daalhuis}, D.~W. Lozier, B.~I. Schneider,
  R.~F. Boisvert, C.~W. Clark, B.~R. Miller and B.~V. Saunders, eds.

\bibitem[FHH18]{ford2018wave}
G~Austin Ford, Andrew Hassell, and Luc Hillairet.
\newblock Wave propagation on euclidean surfaces with conical singularities. i:
  Geometric diffraction.
\newblock {\em Journal of Spectral Theory}, 8(2):605--667, 2018.

\bibitem[Fri80]{friedlander1980radiation}
F~Gerard Friedlander.
\newblock Radiation fields and hyperbolic scattering theory.
\newblock In {\em Mathematical Proceedings of the Cambridge Philosophical
  Society}, volume~88, pages 483--515. Cambridge University Press, 1980.

\bibitem[Fri01]{friedlander2001notes}
F~Gerard Friedlander.
\newblock Notes on the wave equation on asymptotically euclidean manifolds.
\newblock {\em Journal of Functional Analysis}, 184(1):1--18, 2001.

\bibitem[FW17]{ford2017diffractive}
G~Austin Ford and Jared Wunsch.
\newblock The diffractive wave trace on manifolds with conic singularities.
\newblock {\em Advances in Mathematics}, 304:1330--1385, 2017.

\bibitem[H{\"o}r09]{hormander2009analysis}
Lars H{\"o}rmander.
\newblock {\em The analysis of linear partial differential operators IV:
  Fourier integral operators}.
\newblock Springer, 2009.

\bibitem[HV99]{hassell1999spectral}
Andrew Hassell and Andr{\'a}s Vasy.
\newblock The spectral projections and the resolvent for scattering metrics.
\newblock {\em Journal d’Analyse Mathematique}, 79(1):241--298, 1999.

\bibitem[Jos97]{joshi1997intrinsic}
M~Joshi.
\newblock An intrinsic characterisation of polyhomogeneous lagrangian
  distributions.
\newblock {\em Proceedings of the American Mathematical Society},
  125(5):1537--1543, 1997.

\bibitem[Mel94]{melrose1994spectral}
Richard~B Melrose.
\newblock Spectral and scattering theory for the laplacian on asymptotically
  euclidian spaces.
\newblock {\em Lecture Notes in Pure and Applied Mathematics}, pages 85--85,
  1994.

\bibitem[MVW08]{melrose2008propagation}
Richard Melrose, Andr{\'a}s Vasy, and Jared Wunsch.
\newblock Propagation of singularities for the wave equation on edge manifolds.
\newblock {\em Duke Mathematical Journal}, 144(1):109--193, 2008.

\bibitem[MW04]{Melrose-Wunsch1}
Richard Melrose and Jared Wunsch.
\newblock Propagation of singularities for the wave equation on conic
  manifolds.
\newblock {\em Inventiones mathematicae}, 156(2):235--299, 2004.

\bibitem[MZ96]{melrose1996scattering}
Richard Melrose and Maciej Zworski.
\newblock Scattering metrics and geodesic flow at infinity.
\newblock {\em Inventiones Mathematicae}, 124(1-3):389--436, 1996.

\bibitem[SB03]{sa2003radiation}
Ant{\^o}nio S{\'a}~Barreto.
\newblock Radiation fields on asymptotically euclidean manifolds.
\newblock {\em Communications in Partial Differential Equations},
  28(9-10):1661--1673, 2003.

\bibitem[Tay13]{taylor2013partial}
Michael Taylor.
\newblock {\em Partial differential equations II: Qualitative studies of linear
  equations}, volume 116.
\newblock Springer Science \& Business Media, 2013.

\end{thebibliography}
\bibliographystyle{alpha}

\end{document}